\documentclass[11pt]{article}
\usepackage{indentfirst}

\usepackage{float}
\usepackage{booktabs}
\usepackage{makecell}
\usepackage{epsfig}
\usepackage{url}
\usepackage{ulem}
\usepackage{amsmath}
\usepackage{amsfonts}
\usepackage{color}
\usepackage{amssymb}
\usepackage{graphicx,amsfonts,mathrsfs}
\usepackage{epstopdf}
\usepackage{enumerate}
\usepackage[all]{xy}
\usepackage{amsthm}
\usepackage{array}

\def\sm{\rm \setminus}  
\def\id{{\rm \textsf{id}}}
\def\Aut{{\rm \textsf{Aut}}}

\allowdisplaybreaks

\begin{document}
\begin{sloppypar}

\newtheorem{theorem}{Theorem}[section]
\newtheorem{problem}[theorem]{Problem}
\newtheorem{corollary}[theorem]{Corollary}
\newtheorem{definition}[theorem]{Definition}
\newtheorem{conjecture}[theorem]{Conjecture}
\newtheorem{question}[theorem]{Question}

\newtheorem{lemma}[theorem]{Lemma}
\newtheorem{proposition}[theorem]{Proposition}
\newtheorem{quest}[theorem]{Question}
\newtheorem{example}[theorem]{Example}
\newcommand{\remark}{\medskip\par\noindent {\bf Remark.~~}}
\newcommand{\pp}{{\it p.}}
\newcommand{\de}{\em}

\title{  {Some properties of  generalized Cayley graphs}\thanks{This research was supported by  NSFC (No. 12071484).
E-mail addresses: liaoqianfen@163.com(Q.F. Liao),  wjliu6210@126.com(W.J. Liu), zpengli0626@163.com($^\dag$P.-L. Zhang, corresponding author).}}

\author{Qianfen Liao$^1$, Weijun Liu$^{2,3}$, Peng-Li Zhang$^4\dag$\\
{\small $^1$ Department of Mathematics, Guangdong University of Education, } \\
{\small Guangzhou,  Guangdong, 510303, P.R. China.}\\
{\small $^{2}$ School of Mathematics and Statistics, Central South University} \\
{\small New Campus, Changsha, Hunan, 410083, P.R. China. }\\
{\small $^3$ College of General Education, Guangdong University of Science and Technology,} \\
{\small Dongguan, Guangdong, 523083, P.R. China}\\
{\small $^4$ School of Statistics and Information, Shanghai University of International} \\
{\small Business and Economics, Shanghai 201620, P.R. China}
}
\maketitle

\vspace{-0.5cm}

\begin{abstract}
In this paper, firstly, we provide some necessary  and sufficient conditions for generalized Cayley graphs on abelian groups  to be   bipartite. Secondly, we deduce several necessary  and sufficient conditions for generalized Cayley graphs on  finite groups   to be  connected.
At last, as a by-product, we determine the groups whose all cubic generalized Cayley   graphs are connected and integral.
\end{abstract}


\section{Introduction}
All graphs considered in this paper are finite and undirected. Let $X$ be a graph with vertex set $V(X)$ and edge set $E(X).$
For graphs $X$ and $Y$, the {\it direct product} $X \times Y$ of $X$ and $Y$ is the graph with vertex set $V(X\times Y )=\{\{x,y\}\mid x\in V(X), y\in V(Y)\}$ and
two vertices
$(x_1, y_1)$ and $(x_2, y_2)$ are adjacent in $X\times Y$ if and only if $x_1$ is adjacent to $x_2$ in $X$ and $y_1$ is adjacent to $y_2$ in $Y$.
Let $G$ be a finite group.
An element $x$ of $G$ is  a {\it square} if $x=g^2$ for some $g\in G.$
A subset $S$ of $G$ is  {\it square free} if each element in $S$ is not a square.
For two subsets $S_1$ and  $S_2$ of $G$, denote $S_1S_2=\{s_1s_2\mid s_1\in S_1, s_2\in S_2\}$.
Especially, if  $G$ is an abelian group but not an  elementary abelian $2$-group, let $\iota$ be the {\it inverse
automorphism} $g\mapsto g^{-1}$ for all $g\in G$.
A {\it cubic} graph is a regular graph of degree $3.$

The study on generalized Cayley graphs originates from  Maru\v{s}i\v{c} et al. \cite{D.R.N}, who introduced the definition of generalized Cayley graphs as follows and provided some examples of non-Cayley generalized Cayley graphs.

\begin{definition}\label{def1.1}
Let  $G$ be a finite group. For $S\subseteq G$ and $\alpha \in \Aut(G)$, if they satisfy the following  conditions,
\begin{enumerate}[{\rm(a)}]

\item $\alpha^2=\id$, where $\id$ is the identity of $\Aut(G)$;

\item for any $g \in G$, $\alpha(g^{-1}) g\notin S$;

\item for $g,h\in G$, if $\alpha(h^{-1}) g\in S$, then $\alpha(g^{-1}) h\in S$,
\end{enumerate}
then the graph with vertex set $G$ and edge set $\{\{g,h\}\mid \alpha(g^{-1}) h \in S\}$ is denoted by $GC(G,S,\alpha)$.
We call $S$ a generalized Cayley subset of $G$ induced by $\alpha$ and $GC(G,S,\alpha)$ a generalized Cayley graph with respect to the ordered pair $(S, \alpha)$.
Especially, if $\alpha=\id$, then $GC(G, S, \alpha)$ is a Cayley graph.
\end{definition}

In order to differ from Cayley graphs, all automorphisms considered in this paper that inducing generalized Cayley graphs are assumed to be some involutions. Let $G$ be a finite group and $S$ a generalized Cayley subset of $G$ induced by  the involutory automorphism $\alpha$.
For an involutory automorphism $\alpha$, let $\omega_\alpha(G)=\{ \omega_\alpha(g)=\alpha(g^{-1})g \mid g\in G\}$, $\Omega_\alpha(G)=\{g\in G \mid  g \in G\sm \omega_\alpha(G) ~\text{and}~ \alpha(g)=g^{-1}\}$ and $\mho_\alpha(G)=\{g\in G \mid \alpha(g) \neq g^{-1}\}$.
Clearly, the subsets $\omega_\alpha(G)$, $\Omega_\alpha(G)$ and $\mho_\alpha(G)$ form a partition of $G$.
In addition, the conditions $(b)$ and $(c)$ in Definition \ref{def1.1} are equivalent to $S\cap \omega_\alpha(G)=\emptyset$ and $\alpha(S^{-1})=S,$ which implies that,  if $s\in S\cap \mho_\alpha(G)$, then $\alpha(s^{-1})\in S$.
Moreover, if the size of $S$ is odd, then $S\cap \Omega_\alpha(G)\neq\emptyset$.

So far, the research on generalized Cayley graphs is mainly focused on finding examples of non-Cayley generalized Cayley graphs \cite{A.K.D, A.K.P.A}, and  studying isomorphism problems of  generalized Cayley graphs \cite{L,X.W.J.L,X.W.L,Z1}. Recently, Zhu et al. \cite{Z2} studied the spectra of generalized Cayley graphs on finite abelian groups  with the help of  result in \cite{A}, which charaterized the connection between  Cayley graphs and the double cover of generalized Cayley graphs.
As far as we know, these are  all the results  about generalized Cayley graphs up to now.

Let $G$ be an abelian group and $S$ a subset of  $G$.
The Cayley  sum graph $Cay^+(G,S)$~\cite{C.G.W}  is a graph with   vertex set $G$ and   edge set $\{\{g,h\}\mid g,h\in G, gh\in S\}$, where $S$ is a {square free} subset of $G$.
Clearly, if $G$ is an abelian group but not an elementary abelian $2$-group, then the Cayley  sum graph of $G$ is a special generalized Cayley graph of $G$ under the inverse automorphism $\iota$.
Inspired the results of Cayley sum graphs in~\cite{A.T, A.T2, C.G.W}, we are concerned with the necessary and sufficient  conditions for generalized Cayley graphs to be connected and bipartite, respectively.

This paper is organized as follows. In Section~2,
we provide  some necessary and sufficient  conditions for generalized Cayley graphs of abelian groups to be bipartite.
In Section $3$, some necessary and sufficient  conditions for generalized Cayley graphs to be connected are derived.
In Section $4$, we determine the groups whose all cubic generalized Cayley   graphs are connected and integral  by the aid of the results in  Section $3.$

\section{Bipartite generalized Cayley graphs}
In this section, we investigate the conditions for  generalized Cayley graphs of abelian groups to be bipartite.
First, we derive the following property of generalized Cayley graphs on abelian groups.

Let $G$ be an abelian group. For each element $g\in G$, let $R(g): x\mapsto xg$ be a permutation on $G$ and $R(G)=\{R(g)\mid g\in G\}$.
Let $S$ be a generalized Cayley subset of $G$ induced by $\alpha$ and \begin{equation*}
G_\alpha(S)=\{g\in G\mid \alpha(g)Sg^{-1}=S\}=\{g\in G\mid \omega_\alpha(g^{-1})S=S\}.
\end{equation*}
Clearly, $\omega_\alpha(g^{-1})S=S$ is equivalent to $\omega_\alpha(g)S=S$.
Thus, $g\in G_\alpha(S)$ if and only if $g^{-1}\in G_\alpha(S)$.
The following proposition for a generalized Cayley graph  holds.

\begin{proposition}
$R(G)\cap \mathrm{Aut}(\Gamma)=R(G_\alpha(S))$.
\end{proposition}
\begin{proof}
Assume $R(h)$ is an automorphism of $\Gamma$, then for any element $g\in G$ and $s\in S$, $\{gh, \alpha(g)sh\}\in E(\Gamma)$.
It follows that $\alpha(gh)^{-1}\alpha(g)sh=\omega_\alpha(h)s\in S$.
Thus, $h\in G_\alpha(S)$.
Conversely, for any element $g\in G_\alpha(S)$, we claim that $R(g)\in \mathrm{Aut}(\Gamma)$.
Observe that
\begin{equation*}
\omega_\alpha(g)\alpha(g_1^{-1})g_2=\alpha(g^{-1})g\alpha(g_1^{-1})g_2=
\alpha((g_1g)^{-1})(g_2g).
\end{equation*}
Since both $g$ and $g^{-1}$ are contained in $G_\alpha(S)$,  $\alpha(g_1^{-1})g_2\in S$ if and only if $\alpha((g_1g)^{-1})(g_2g)\in S$.
Therefore, $R(g)\in \mathrm{Aut}(\Gamma)$.
\end{proof}


\begin{lemma}\label{lem1.2}
Let $\alpha$ be an involutory automorphism of an abelian group $G$.
If $g_1g_2\in \omega_\alpha(G)$ and $g_1\in \omega_\alpha(G)$, then $g_2\in \omega_\alpha(G)$.
Moreover, if $h_1,h_2\in \omega_\alpha(G)$, then $h_1^nh_2^m\in \omega_\alpha(G)$ for all integers $n$ and $m$.
\end{lemma}
\begin{proof}
Assume $g_1g_2=\alpha(a_1^{-1})a_1$ and $g_1=\alpha(a_2^{-1})a_2$ for some $a_1,a_2\in G$.
Then $g_2=\alpha(a_1^{-1}a_2)a_1a_2^{-1}\in \omega_\alpha(G)$.
If $h_1=\alpha(b_1^{-1})b_1$ and $h_2=\alpha(b_2^{-1})b_2$ for some $b_1,b_2\in G$,
then for any integers $n$ and $m$, $h_1^nh_2^m=\alpha(b_1^{-n}b_2^{-m})b_1^nb_2^m\in \omega_\alpha(G)$.
\end{proof}

\begin{lemma}\label{lem1.3}
Let $H$ be a graph.
If  $H$ contains a closed walk of odd length, then $H$ contains an odd cycle.
\end{lemma}
\begin{proof}
Let $W=v_0e_1v_1e_2\cdots v_{k-1}e_kv_0$ be a shortest odd-length closed walk of $H$, where $k$ is odd.
If there exist vertices $v_i$ and $v_j$ such that $i\neq j$ and $v_i=v_j$,
then we obtain two closed walks $W_1=v_ie_{i+1}\cdots e_jv_j$ and $W_2=v_0e_1\cdots v_{i-1}e_iv_ie_{j+1}v_{j+1}\cdots v_{k-1}e_kv_0$.
Note that one of the walks $W_1$ and $W_2$ is odd-length, which contradicts the choice of $W$.
Thus, the vertices in $W$ are all distinct, which implies $W$ is an odd cycle.
\end{proof}

During the proof,  the following  fact  will be  used frequently.\\
{\bf Fact}: For any  $s\in S\cap \Omega_\alpha(G)$, we have $s^2=\alpha(s^{-1})s\in \omega_\alpha(G)$ and further  $s^{2k}\in \omega_\alpha(G)$ for any integer $k$.

\begin{theorem}\label{thm2.6}
Let $GC(G, S, \alpha)$ be a generalized Cayley graph.
Then $GC(G,S,\alpha)$ is not bipartite if and only if for each $s\in S$, there exists an non-negative integer $k_s$ such that $\Sigma_{s\in S}k_s$ is odd
 and $\prod _{s\in S}s^{k_s}\in \omega_\alpha(G)$.
\end{theorem}
\begin{proof}
Necessity.
Since $GC(G,S,\alpha)$ is not bipartite, it contains an odd cycle $C_{2k+1}=(x_0, x_1, \cdots, x_{2k}, x_{2k+1}=x_0)$, which implies that $\alpha(x_i^{-1})x_{i+1}=s_i\in S$ for each integer $0\leq i\leq 2k$.
Then $\alpha((x_0\cdots x_{2k})^{-1})(x_0\cdots x_{2k})=s_0\cdots s_{2k}\in \omega_\alpha(G)$.

Sufficiency.
Assume that for each $s\in S$, there exists an integer  $k_s$ such that $\Sigma_{s\in S}k_s=2t+1$ is odd  and $\prod _{s\in S}s^{k_s}\in \omega_\alpha(G)$.
Let $T=\{s_1, s_2, \cdots, s_{2t+1}\}$, where for $1\leq j\leq 2t+1$, $s_j$ are not necessary distinct and  for each $s\in S$, $s$ appears in $T$ exactly $k_s$ times.
Then $\prod _{s\in S}s^{k_s}=\prod _{j=1}^{2t+1}s_j$.
We divide $T$ into two subsets $T_1$ and $T_2$ such that the size of $T_1$ and $T_2$ are $t+1$ and $t$, respectively.
Let $T_1=\{s_1, s_2, \cdots, s_{t+1}\}$ and $T_2=\{s_{t+2}, \cdots, s_{2t+1}\}$.
Since $\prod _{j=1}^{2t+1}s_j\in \omega_\alpha(G)$ and $s_\ell^{-1}\alpha(s_\ell)\in \omega_\alpha(G)$ for $t+2\leq \ell\leq 2t+1$,
by Lemma \ref{lem1.2}, we have
\begin{equation*}
\prod _{j=1}^{2t+1}s_j \prod _{\ell=t+2}^{2t+1}\left(s_\ell^{-1}\alpha(s_\ell)\right)=\prod _{j=1}^{t+1}s_j\prod _{\ell=t+2}^{2t+1}\alpha(s_\ell)\in \omega_\alpha(G).
\end{equation*}
Suppose $\prod _{j=1}^{t+1}s_j\prod _{\ell=t+2}^{2t+1}\alpha(s_\ell)=\alpha(g^{-1})g$ for some $g\in G$.
Then
\begin{equation*}
(g, \alpha(g)s_1, g\alpha(s_1)s_{t+2}, \cdots, g\prod _{j=1}^{t}\alpha(s_j)\prod _{j=1}^{t}s_{t+1+j},\alpha(g)\prod _{j=1}^{t+1}s_j\alpha\left(\prod _{j=1}^{t}s_{t+1+j}\right)=g)
\end{equation*}
is an odd closed walk in $GC(G, S, \alpha)$.
By Lemma \ref{lem1.3}, $GC(G, S, \alpha)$ contains an odd cycle and then it is not a bipartite graph.
\end{proof}

The following  two examples are applications of Theorem \ref{thm2.6}.

\begin{example}
For the cyclic group $Z_{14}=\langle g\rangle$, $S_1=\{g, g^3, g^5\}$ is a generalized Cayley subset of $Z_{14}$ induced by the inverse automorphism $\iota$ and contained in $\Omega_\iota(Z_{14})$.
The graph $GC(Z_{14}, S_1, \iota)$ is displayed in Figure \ref{fig4}.
Observe that for any non-negative integers $k_1$, $k_2$ and $k_3$ with $k_1+k_2+k_3$ being odd, $k_1+3k_2+5k_3$ being odd.
Since $\omega_\iota(Z_{14})=\langle g^2\rangle$, $g^{k_1}g^{3k_2}g^{5k_3}=g^{k_1+3k_2+5k_3}\notin \omega_\iota(Z_{14})$.
It follows from Theorem \ref{thm2.6}  that $GC(Z_{14}, S_1, \iota)$ is bipartite.
\begin{figure}[h]
\centering
\includegraphics[scale=0.5]{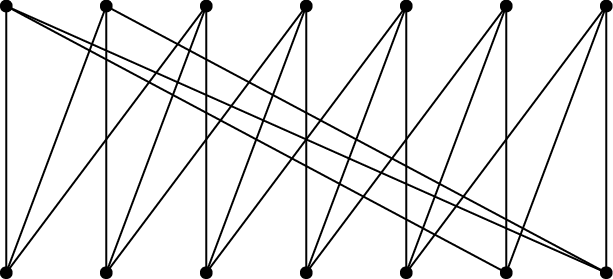}
\caption{$GC(Z_{14}, S_1, \iota)$}
\label{fig4}
\end{figure}
\end{example}

\begin{example}
For the group $\mathbb{Z}_2^2\times \mathbb{Z}_6$, $S_2=\{(1, 0, 0), (0, 0, 2), (0, 0, 4)\}$ is
a generalized Cayley subset of $\mathbb{Z}_2^2\times \mathbb{Z}_6$ induced by the involutory automorphism $\alpha: (1, 0, 0)\mapsto (1, 0, 0), (0, 1, 0)\mapsto (0, 1, 0), (0, 0, 1)\mapsto (0, 1, 1)$, where $(1, 0, 0)\in \Omega_\alpha(\mathbb{Z}_2^2\times \mathbb{Z}_6)$ and $\{(0, 0, 2), (0, 0, 4)\}\subseteq \mho_\alpha(\mathbb{Z}_2^2\times \mathbb{Z}_6)$.
Since $(0, 0, 2)^3=(0, 0, 0)\in \omega_\alpha(\mathbb{Z}_2^2\times \mathbb{Z}_6)$, it follows from Theorem \ref{thm2.6} that $GC(\mathbb{Z}_2^2\times \mathbb{Z}_6, S_2, \alpha)$ is not bipartite.
The graph $GC(\mathbb{Z}_2^2\times \mathbb{Z}_6, S_2, \alpha)$ is displayed in Figure \ref{fig5}.
\begin{figure}[h]
\centering
\includegraphics[scale=0.5]{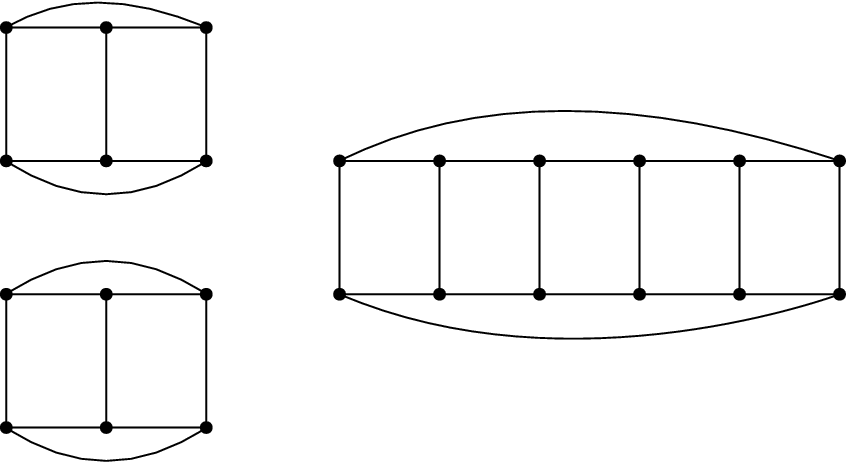}
\caption{$GC(\mathbb{Z}_2^2\times \mathbb{Z}_6, S_2, \alpha)$}
\label{fig5}
\end{figure}
\end{example}

\section{Connected generalized Cayley graphs}
In this section, we  provide the necessary and sufficient  conditions  for the generalized Cayley graphs to be connected.
We assume that $G$ is a finite group, $S$ is a generalized Cayley subset of $G$ induced by the involutory automorphism $\alpha$.
Firstly, a simple observation should be pointed out.

\begin{proposition}\label{pro3.0}
Let $GC(G, S, \alpha)$ be a generalized Cayley graph. Then
$\{x,y\}\in E(Cay(G, S^{-1}S))$ if and only if there exists a walk of length $2$ from $x$ to $y$ in $GC(G, S, \alpha)$.
\end{proposition}
\begin{proof}
Notice that $e\in S^{-1}S$, it follows that $Cay(G, S^{-1}S)$ is an undirected  graph {with loops}.
If $\{x,y\}\in E(Cay(G, S^{-1}S))$, then there exist two elements $a,b\in S$ such that $y=x(a^{-1}b)$, which implies $\alpha(y)\alpha(b^{-1})=\alpha(x)\alpha(a^{-1})$.
Note that $\alpha(S^{-1})=S$, hence both $\alpha(a^{-1})$ and $\alpha(b^{-1})$ belong to $S$.
Therefore, $(y,\alpha(y)\alpha(b^{-1}), x)$ is a walk of length $2$ from $x$ to $y$ in $GC(G, S, \alpha)$.

Conversely, if there exists a walk $(x, \alpha(x)s_1, y)$ from $x$ to $y$, where $s_1\in S$,
then $y=\alpha(\alpha(x)s_1)s_2=x(\alpha(s_1)s_2)$ for some $s_2\in S$.
Recall $\alpha(S^{-1})=S$, we have  $\alpha(s_1)s_2\in S^{-1}S$.
Thus,  $\{x,y\}\in E(Cay(G, S^{-1}S))$.
\end{proof}

For the  generalized Cayley graph $GC(G,S,\alpha)$, we define an equivalence relation on $G$ by declaring  that $x$ is equivalent to $y$ if there exists a walk of even length from $x$ to $y$ in $GC(G,S,\alpha)$.
Denote by $\Theta$  the equivalence class containing the identity element $e$ of $G$.
Using this equivalence relation,  we obtain the following result.
\begin{lemma}\label{lem3.0}
Let $GC(G, S, \alpha)$ be a generalized Cayley graph.
If  $GC(G,S,\alpha)$ is connected, then $\Theta=\langle S^{-1}S\rangle$ and $|\langle S^{-1}S\rangle|\geq \frac{1}{2}|G|$.
\end{lemma}
\begin{proof}
Firstly, we  show that $\Theta\subseteq \langle S^{-1}S\rangle$ and the size of $\Theta$ is not smaller than $\frac{1}{2}|G|$.
For any $y\in \Theta$, there exists a walk of even length from $e$ to $y$.
Then by Proposition \ref{pro3.0}, there exists a walk $(e, y_1, y_2, \cdots, y_m=y)$ from $e$ to $y$ in $Cay(G, S^{-1}S)$.
Observe that $y_1\in S^{-1}S$, then $y_2=y_1x_1\in \langle S^{-1}S\rangle$ for some $x_1\in S^{-1}S$.
Similarly, we obtain $y_i\in \langle S^{-1}S\rangle$ for each integer $2\leq i\leq m$.
Thus, $\Theta\subseteq \langle S^{-1}S\rangle$.
Since $GC(G, S,\alpha)$ is connected, for each element $z\in G\setminus \Theta$, let $(e, z_1, z_2, \cdots, z_{2t+1}=z)$ be a walk of odd length from $e$ to $z$.
Let $s\in S$.
Then $(e, z_1, z_2, \cdots, z_{2t+1}, \alpha(z)s)$ is a walk of even length from $e$ to $\alpha(z)s$, which implies $\alpha(z)s\in \Theta$.
It follows  that each element $z$ in $G\setminus \Theta$ corresponds to the element $\alpha(z)s$ in $\Theta$.
Moreover, for two distinct elements $z_1$ and $z_2$ in $G\setminus \Theta$, $\alpha(z_1)s\neq \alpha(z_2)s$.
Hence, we have the size of $\Theta$ is not smaller that $\frac{1}{2}|G|$.
As a consequence, only one of the two conditions $\Theta=\langle S^{-1}S\rangle$ and, $\Theta\subsetneq\langle S^{-1}S\rangle$ and $\langle S^{-1}S\rangle=G$ holds.
In fact, the latter case is impossible.
If $\langle S^{-1}S\rangle=G$, then $Cay(G, S^{-1}S)$ is connected.
Thus, for any vertex $x\in G$, there is a path from $x$ to $e$ in $Cay(G, S^{-1}S)$.
By Proposition \ref{pro3.0},  there is an even walk from  $x$ to $e$ in $GC(G, S, \alpha)$, which yields that  $\Theta=G=\langle S^{-1}S\rangle$.
Therefore, the desired result follows.
\end{proof}

To prove Theorem \ref{thm3.3}, the next two lemmas will also be needed.

\begin{lemma}\label{lem3.1}
Let $GC(G, S, \alpha)$ be a generalized Cayley graph. It holds that
$\alpha(\langle S^{-1}S\rangle)=\langle SS^{-1}\rangle$.
\end{lemma}
\begin{proof}
Note that $(S^{-1}S)^{-1}=S^{-1}S$, hence each element of $\langle S^{-1}S\rangle$ is the product of some elements in $S^{-1}S$.
Since $\alpha(S^{-1}S)=\alpha(S^{-1})\alpha(S)=SS^{-1}$, we have $\alpha(\langle S^{-1}S\rangle)\leq\langle SS^{-1}\rangle$.
Similarly, for $SS^{-1}$, we have $\alpha(\langle SS^{-1}\rangle)\leq\langle S^{-1}S\rangle$.
Since $\alpha$ is an involutory automorphism of $G$, it follows that $\langle SS^{-1}\rangle\leq\alpha(\langle S^{-1}S\rangle)$.
Thus, $\alpha(\langle S^{-1}S\rangle)=\langle SS^{-1}\rangle$.
\end{proof}

\begin{lemma}\label{lem3.2}
Let $GC(G, S, \alpha)$ be a generalized Cayley graph.
Let $C(e)$ be the vertex set  of the connected component containing $e$ in $GC(G, S, \alpha)$.
Then  $C(e)=\langle S^{-1}S\rangle\cup \langle SS^{-1}\rangle s$ for any $s\in S$.
\end{lemma}
\begin{proof}
 Notice that for any $s\in S$, $t=ts^{-1}s\in \langle SS^{-1}\rangle s$ for every $t\in S$.
It follows that  $S\subseteq \langle SS^{-1}\rangle s$ and $\langle SS^{-1}\rangle t= \langle SS^{-1}\rangle s$.

On the one hand, for each vertex $x\in C(e)$, there exists a path $P=(e,x_1,x_2,\cdots,x_n=x)$ from $e$ to $x$.
We claim that either $x_i\in \langle S^{-1}S\rangle$ or $x_i\in\langle SS^{-1}\rangle s$ for $1\leq i\leq n$.
It is easy to see that $x_1=\alpha(e)s_1=s_1\in \langle SS^{-1}\rangle s$ and then $x_2=\alpha(x_1)s_2\in \alpha(\langle SS^{-1}\rangle)\alpha(s)s_2$ for some $s_1,s_2\in S$.
By Lemma \ref{lem3.1} and $\alpha(S)=S^{-1}$, we have $x_2\in \langle S^{-1}S\rangle$.
By induction, suppose that for each integer $t$ smaller than $k$, where $k\leq n$,  $x_t\in \langle SS^{-1}\rangle s$  when $t$ is odd and $x_t\in\langle S^{-1}S\rangle$ when $t$ is even.
Then from $x_{k-1}\in \langle SS^{-1}\rangle s\cup\langle S^{-1}S\rangle$ and $x_k=\alpha(x_{k-1})s_k$ for some $s_k\in S$, it follows that $x_k\in \langle S^{-1}S\rangle\cup \langle SS^{-1}\rangle s$.
Thus, $C(e)\subseteq \langle S^{-1}S\rangle\cup \langle SS^{-1}\rangle s$.

On the other hand, we claim that  $\langle S^{-1}S\rangle\cup \langle SS^{-1}\rangle s\subseteq C(e)$.
For any $y\in \langle S^{-1}S\rangle$, assume that $y=\prod_{i=1}^{m}s_i^{-1}t_i$, where $s_i,t_i\in S$ for $1\leq i\leq m$.
Since $\alpha(S)=S^{-1}$, there exists $u_i\in S$ such that $\alpha(u_i)=s_i^{-1}$ for each  $1\leq i\leq m$.
Then \begin{equation*}
W_1=\left(e,u_1,s_1^{-1}t_1,u_1\alpha(t_1)u_2,\prod_{i=1}^2s_i^{-1}t_i,\cdots,
\left(\prod_{i=1}^{m-1}u_i\alpha(t_i)\right)u_m,\prod_{i=1}^{m}s_i^{-1}t_i=
y\right)
\end{equation*}
is a walk from $e$ to $y$.
For any $y\in \langle SS^{-1}\rangle s$, assume $y=(\prod_{j=1}^{n}s_jt_j^{-1})s$, where $s_j,t_j\in S$ for $1\leq j\leq n$.
For each $1\leq j\leq n$, let $v_j=\alpha(t_j^{-1})\in S.$
Then
\begin{equation*}
W_2=\left(e,s_1,\alpha(s_1)v_1,s_1t_1^{-1}s_2,
\prod_{j=1}^{2}\alpha(s_j)v_j,\cdots,
\prod_{j=1}^{n}\alpha(s_j)v_j, \left(\prod_{j=1}^{n}s_jt_j^{-1}\right)s=y\right)
\end{equation*}
is a walk from $e$ to $y$, which deduces that $\langle SS^{-1}\rangle s\subseteq C(e).$

Therefore,   $C(e)=\langle S^{-1}S\rangle\cup \langle SS^{-1}\rangle s$ follows.
\end{proof}

Combining with Lemmas~\ref{lem3.1} and~\ref{lem3.2}, we deduce the following result.

\begin{theorem}\label{thm3.3}
Let $GC(G, S, \alpha)$ be a generalized Cayley graph. Then
$GC(G,S,\alpha)$ is connected if and only if one of the following conditions holds:
\begin{enumerate}[{\rm(i)}]

\item $S$ is not contained in a right coset of any proper subgroup of $G$;

\item $S$ is contained in a nontrivial right coset of  a proper subgroup $H$ of $G$, where $H=\langle SS^{-1}\rangle$  and it  satisfies $|G:H|=2$ and $\alpha(H)=H$.

\end{enumerate}
\end{theorem}
\begin{proof}
Suppose $GC(G,S,\alpha)$ is connected and let $H$ be the smallest subgroup of $G$ such that $S$ is contained in some right coset of $H$.
Assume $S\subseteq Hg$ for some $g\in G$.
Then  $S^{-1}\subseteq g^{-1}H$, which yields that  $SS^{-1}\subseteq Hgg^{-1}H=H$.
Thus, $\langle SS^{-1}\rangle\leq H$.
Since for any $s\in S$, we have $S\subseteq \langle SS^{-1}\rangle s$, which implies that $S$ is also contained in a right coset of $\langle SS^{-1}\rangle$.
By the choice of $H$,  $\langle SS^{-1}\rangle=H$ follows.
From Lemma \ref{lem3.0}, either $H=G$ or  $|H|=\frac{1}{2}|G|$.
For the latter case, $H$ is a normal subgroup of $G$.
It follows that $S^{-1}S\subseteq g^{-1}HHg=H$ and then $\langle S^{-1}S\rangle\leq H$.
By Lemma \ref{lem3.1}, we obtain $\alpha(H)=\alpha(\langle SS^{-1}\rangle)=\langle S^{-1}S\rangle\leq H$.
Since the order of $\alpha(H)$ and $H$ are  equal, we have $\alpha(H)= H$.
In addition, we assert that $g\notin H$.
Otherwise, assume $g\in H$, then $S\subseteq H$.
From  Lemma \ref{lem3.2}, we obtain the connected component $C(e)$ containing $e$ in $GC(G,S,\alpha)$ satisfies $C(e)\leq\langle S\rangle\leq H$, which contradicts that $GC(G,S,\alpha)$ is connected.

Conversely, we show that both the conditions \rm(i) and \rm(ii)  imply $GC(G,S,\alpha)$ is connected.
Let $s\in S$.
If $S$ is not contained in a  right coset of any proper subgroup of $G$,
then we have $G=\langle SS^{-1}\rangle$ since $S\subseteq \langle SS^{-1}\rangle s$.
Thus, by Lemma \ref{lem3.2}, $C(e)=G$, which implies that  $GC(G,S,\alpha)$ is connected.
Let $H=\langle SS^{-1}\rangle<G$.
If $S$ is contained in $Hg$ for some $g\in G\setminus H$, where $\alpha(H)=H$ and $|G:H|=2$, then $C(e)=\alpha(H)\cup Hs=H\cup Hs=H\cup Hg=G$.
Thus, $GC(G,S,\alpha)$ is  connected.
\end{proof}

By Theorem \ref{thm3.3}, we obtain  another equivalent condition for generalized Cayley graphs to be connected,
which is more convenient to use in general.

\begin{theorem}\label{thm3.4}
Let $GC(G, S, \alpha)$ be a generalized Cayley graph. Then
$GC(G,S,\alpha)$ is connected if and only if $G=\langle S\rangle$, $|G:\langle SS^{-1}\rangle|\leq 2$ and $\alpha(\langle SS^{-1}\rangle)=\langle SS^{-1}\rangle$.
\end{theorem}
\begin{proof}
Necessity.
Assume $GC(G,S,\alpha)$ is connected.
Since the connected component $C(e)$ containing $e$ in $GC(G,S,\alpha)$ satisfies $C(e)=\langle S^{-1}S\rangle\cup \langle SS^{-1}\rangle s\subseteq \langle S\rangle$ and $C(e)=G$ for any $s\in S$, we have $G=\langle S\rangle$.
Note that $S\subseteq \langle SS^{-1}\rangle s$,
by Theorem \ref{thm3.3},
$\langle SS^{-1}\rangle=G$ or  $\alpha(\langle SS^{-1}\rangle)=\langle SS^{-1}\rangle$ and $|G:\langle SS^{-1}\rangle|=2$.

Sufficiency.
We split the proof into two cases:
\vspace{2mm}

\textbf{Case 1:} $G=\langle S\rangle=\langle SS^{-1}\rangle=\langle S^{-1}S\rangle$.

\vspace{2mm}
In this case, we claim that there is no proper subgroup of $G$ such that $S$ is contained in a right coset of this subgroup.
If the claim is not true, let $H<G$ and $S\subseteq Hg$ for some $g\in G$.
Then $SS^{-1}\subseteq Hgg^{-1}H=H$.
Further, $G=\langle SS^{-1}\rangle \leq H$, which is a contradiction.
Thus, by  Theorem \ref{thm3.3} (i), the result follows.

\vspace{2mm}

\textbf{Case 2:} $G=\langle S\rangle$, $|G:\langle SS^{-1}\rangle|=2$ and  $\alpha(\langle SS^{-1}\rangle)=\langle SS^{-1}\rangle$.

\vspace{2mm}
If there exists a proper subgroup $H$ of $G$ such that $S\subseteq Hg$ for some $g\in G$,
then we obtain $\langle SS^{-1}\rangle\leq H$.
Since  $|G:\langle SS^{-1}\rangle|=2$, it follows that $H=\langle SS^{-1}\rangle$.
Clearly, $g\notin H$.
Otherwise, $S\subseteq H$ and then $\langle S\rangle=G\leq H$, which is a contradiction.
Using Theorem \ref{thm3.3} (ii), we conclude that $GC(G,S,\alpha)$ is connected.
\end{proof}

According to Theorem \ref{thm3.4}, we give the equivalent conditions of the connected generalized Cayley graphs to be  bipartite.

\begin{theorem}
Let $GC(G, S, \alpha)$ be a generalized Cayley graph.
If $GC(G,S,\alpha)$ is connected, then the following statements are equivalent:
\begin{enumerate}[{\rm(i)}]
\item $GC(G,S,\alpha)$ is bipartite;

\item $|\langle SS^{-1}\rangle|=\frac{1}{2}|G|$;

\item $S\cap \langle SS^{-1}\rangle=\emptyset$.
\end{enumerate}
\end{theorem}
\begin{proof}
As $GC(G,S,\alpha)$ is connected, by Theorem \ref{thm3.4}, we have
$\langle S\rangle=G$,  $|G:\langle SS^{-1}\rangle|\leq 2$ and $\alpha(\langle SS^{-1}\rangle)=\langle S^{-1}S\rangle=\langle SS^{-1}\rangle$.
Next, we show that the conditions $\rm(i)$, $\rm(ii)$ and $\rm(iii)$ are equivalent.

$\rm(i)\Rightarrow (ii):$ Since $GC(G,S,\alpha)$ is connected, by Lemma \ref{lem3.0}, for each element $x\in \langle SS^{-1}\rangle$, there exists an even walk from $e$ to $x$.
Combining  $GC(G,S,\alpha)$ is bipartite, it follows that $\langle SS^{-1}\rangle$ is contained in  the same  partition set of $GC(G,S,\alpha)$.
Thus, $\langle SS^{-1}\rangle\neq G$.

$\rm(ii)\Rightarrow (iii):$ Suppose that $|\langle SS^{-1}\rangle|=\frac{1}{2}|G|$ but $S\cap \langle SS^{-1}\rangle\neq\emptyset$.
Let $a\in S\cap \langle SS^{-1}\rangle$.
Then, for any $b\in S$, $b=(ba^{-1})a\in \langle SS^{-1}\rangle$, which yields that $S\subseteq \langle SS^{-1}\rangle$.
As a result, $\langle S\rangle\neq G$, which leads to a contradiction.

$\rm(iii)\Rightarrow (i):$ If $S\cap \langle SS^{-1}\rangle=\emptyset$, then $\langle SS^{-1}\rangle\neq G$, which implies $|\langle SS^{-1}\rangle|=\frac{1}{2}|G|$.
For any edge $\{x,y\}$ in $GC(G,S,\alpha)$, we have $y=\alpha(x)s$ for some $s\in S$.
Observe that $s\not\in \langle SS^{-1}\rangle$, thus $G=\langle SS^{-1}\rangle \cup \langle SS^{-1}\rangle s$.
Without loss of generality, assume that $x\in \langle SS^{-1}\rangle$.
Then $y=\alpha(x)s\in \alpha(\langle SS^{-1}\rangle) s=\langle SS^{-1}\rangle s$.
Thus, $GC(G,S,\alpha)$ is bipartite with bipartition $\langle SS^{-1}\rangle$ and $\langle SS^{-1}\rangle s$.
\end{proof}
{\bf Remark:}
If $G$ is an abelian group, then  $\alpha(\langle SS^{-1}\rangle)=\langle SS^{-1}\rangle$ always holds.
Thus, the generalized Cayley graph of an abelian group $G$ with respect to the pair $(S, \alpha)$ is connected if and only if $G=\langle S\rangle$ and  $|G:\langle SS^{-1}\rangle|\leq 2$.

\section{Connected   integral cubic generalized Cayley graphs}
A graph is  integral if  all eigenvalues of its adjacency matrix are integers. In this section,
 combining with the results about generalized Cayley graphs in Section $3$,  we determine the groups whose all cubic generalized Cayley graphs  are connected and integral.


\begin{lemma}[\cite{A.K.P.A}]\label{lem4.0}
Let $GC(G,S,\alpha)$ be a generalized Cayley graph of a finite group $G$.
Then $GC(G,S,\alpha)\cong GC(G, S^\beta, \alpha^\beta)$ for all $\beta\in \Aut(G)$, where $\alpha^\beta=\beta\alpha \beta^{-1}$.
\end{lemma}

By Lemma \ref{lem4.0}, to study all generalized Cayley graphs of a finite group $G$, it suffices to study the generalized Cayley graphs induced by the involutory automorphisms in different conjugate classes of $\Aut(G)$.

\begin{lemma}[\cite{B.C}]\label{lem4.2}
A  connected integral cubic graph has $n$ vertices, where $n\in \{4,6,8,10,12,20,24,30\}$.
Moreover, all cubic, connected, integral  graphs are displayed in Figure \ref{fig1}.
\end{lemma}
\begin{figure}[t]
\centering
\includegraphics[scale=0.6]{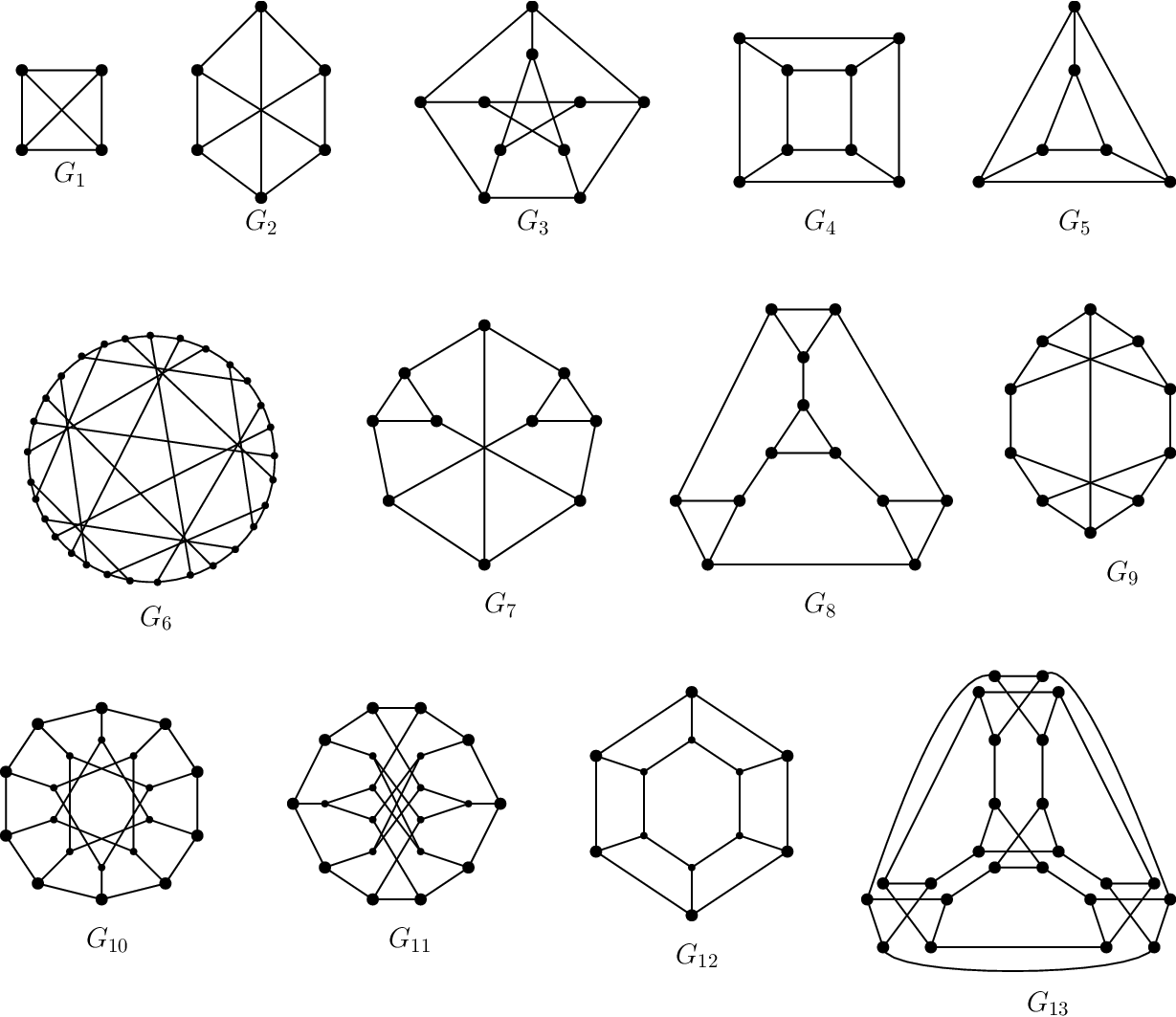}
\caption{All connected, integral, cubic  graphs}
\label{fig1}
\end{figure}

\begin{lemma}[\cite{A.T2}]\label{lem4.3}
Let $G$ be an abelian group with square free subset.
For every square free subset $S$ of $G$ of size $3$, $Cay^+(G,S)$ is connected and integral if and only if $G$ is isomorphic to one of the groups $\mathbb{Z}_2^2$, $\mathbb{Z}_2^3$, $\mathbb{Z}_6$ and $\mathbb{Z}_8$.
\end{lemma}

\begin{lemma}[\cite{W.S}]\label{lem4.11}
If the direct product $X\times Y$ of graphs  $X$ and $Y$ is connected, then both $X$ and  $Y$ are connected.
\end{lemma}

\begin{lemma}[\cite{A.K.P.A}]\label{lem4.12}
Let $X=GC(G_1,S_1,\alpha_1)$, $Y=GC(G_2,S_2,\alpha_2)$, and let $\alpha_1\times \alpha_2$ be the automorphism of $G_1\times G_2$ defined by $(\alpha_1\times \alpha_2)(g_1,g_2)=(\alpha_1(g_1),\alpha_2(g_2))$ for all $g_1\in G_1$ and $g_2\in G_2$.
Then $X \times Y= GC(G_1\times G_2, S_1\times S_2, \alpha_1\times \alpha_2)$.
\end{lemma}

Based on the above lemmas, we determine the abelian groups whose all cubic generalized Cayley graphs are connected and integral.

\begin{theorem}\label{thm4.6}
Let $G$ be a finite abelian group admitting cubic generalized Cayley graph.
Then all cubic generalized Cayley graphs of $G$ are connected and integral if and only if $G\cong  \mathbb{Z}_6$, $\mathbb{Z}_2^3$ or $\mathbb{Z}_8$.
\end{theorem}
\begin{proof}
\textcolor{red}{
According to Lemma \ref{lem4.2}, to determine the abelian groups whose all cubic generalized Cayley graphs are connected and interal, it suffices to consider the abelian groups with order $n\in \{4,6,8,10,12,20,24,30\}$.
Let $G$ be an abelian group with order $n\in\{4, 6, 8, 10, 12, 20, 24, 30\}$.
Then $G$ can be divided into three cases.}

\textbf{Case~1}: $G$ is an elementary abelian group.
In this case, $G\cong \mathbb{Z}_2^2$ or $G\cong\mathbb{Z}_2^3$.

$\bullet$ For  $\mathbb{Z}_2^2$,  all involutory automorphisms of $\mathbb{Z}_2^2$ are in the same conjugate class.
    Let $\delta: (0,1)\mapsto (1,0), (1,0)\mapsto (0,1)$ be an involutory automorphism of $\mathbb{Z}_2^2$ and  we obtain  subsets  $\omega_\delta(\mathbb{Z}_2^2)=\{(0,0),(1,1)\}$, $\Omega_\delta(\mathbb{Z}_2^2)=\emptyset$ and $\mho_\delta(\mathbb{Z}_2^2)=\{(0,1), (1,0)\}$.
    Thus, there is no cubic generalized Cayley graph of $\mathbb{Z}_2^2$.

$\bullet$ For $\mathbb{Z}_2^3$, $\Aut(\mathbb{Z}_2^3)\cong GL(3,2)$ and all involutory automorphisms  of $Z_2^3$ are in the same conjugate class.
By Lemma \ref{lem4.0}, it is enough to discuss the generalized Cayley graphs induced by  $\alpha: (1,0,0)\mapsto (1,0,0), (0,1,0)\mapsto (0,1,0), (0,0,1)\mapsto (0,1,1)$.
With respect to $\alpha$, we obtain three subsets  $\omega_\alpha(\mathbb{Z}_2^3)=\{(0,0,0), (0,1,0)\}$, $\Omega_\alpha(\mathbb{Z}_2^3)=\{(1,0,0), (1,1,0)\}$ and $\mho_\alpha(\mathbb{Z}_2^3)=\{(0,0,1), (0,1,1), (1,0,1), (1, 1, 1)\}$.
Thus, the generalized Cayley subsets of size $3$ induced by $\alpha$ are $ S_1=\{(1,0,0),(0,0,1),(0,1,1)\}$, $S_2=\{(1,0,0),(1,0,1),(1,1,1)\}$, $S_3=\{(1,1,0),(0,0,1),(0,1,1)\}$ and $S_4=\{(1,1,0),(1,0,1),(1,1,1)\}$.
It is easy to verify that for $1\leq i\leq 4$, $\langle S_i\rangle=\mathbb{Z}_2^3$ and $|\langle S_iS_i^{-1}\rangle|\geq \frac{1}{2}|\mathbb{Z}_2^3|$.
By Theorem \ref{thm3.4}, all cubic generalized Cayley graphs are connected.
Moreover, for each $1\leq i\leq 4$, $GC(\mathbb{Z}_2^3, S_i, \alpha)\cong G_4,$ which is displayed in Figure \ref{fig1}.

\textcolor{red}{
\textbf{Case~2}: $G$ is not an elementary abelian $2$-group  without square free subset of size $3$.}

\textcolor{red}{
From the proof of Theory 3.18 in  \cite{A.T2}, it follows that $G\cong \mathbb{Z}_4$.
In this case, there is no cubic generalized Cayley graph of  $\mathbb{Z}_4$.}

\textcolor{red}{
\textbf{Case~3}: $G$ is not an elementary abelian $2$-group  with  square free subsets of size $3$.}

\textcolor{red}{
In this case, each Cayley sum graph of $G$ is a special generalized Cayley graph induced by the inverse automorphism.
Thus, all cubic generalized Cayley graphs of $G$ are connected and integral implies that all cubic Cayley sum graphs of $G$ are connected and integral.}
From Lemma \ref{lem4.3}, it suffices to consider the groups $\mathbb{Z}_6$ and $\mathbb{Z}_8$.

$\bullet$ For $\mathbb{Z}_6$, it has the unique involutory automorphism $\iota$, which implies the generalized Cayley graphs of $\mathbb{Z}_6$ are Cayley sum graphs.

$\bullet$ For $\mathbb{Z}_8$, the involutory automorphisms of which are $\iota$, $\alpha_1: g\mapsto g^3$ and  $\alpha_2:g\mapsto g^5$ for all $g\in \mathbb{Z}_8$.
Since $\Omega_{\alpha_1}(\mathbb{Z}_8)=\Omega_{\alpha_2}(\mathbb{Z}_8)=\emptyset$, there is no cubic generalized Cayley graph induced by $\alpha_1$ and $\alpha_2$.
Thus, all generalized Cayley graphs of $\mathbb{Z}_8$ are Cayley sum graphs.
It follows that all cubic generalized Cayley graphs of $\mathbb{Z}_6$ and  $\mathbb{Z}_8$ are connected and integral.

Therefore, we complete the proof.
\end{proof}

Next, we aim to determine the non-abelian groups whose all cubic generalized Cayley graphs are connected and integral.

\begin{lemma}\label{lem4.13}
For any integer $n\geq 2$, there exists a cubic generalized Cayley graph of the dicyclic group $T_{4n}$ which is not connected.
\end{lemma}
\begin{proof}
Let $T_{4n}=\langle a, b \mid a^{2n}=b^4=e, a^n=b^2, b^{-1}ab=a^{-1}\rangle$.
Note that $\alpha: a\mapsto a, b\mapsto a^nb$ is an involutory automorphism of $T_{4n}$ and $S=\{b, a^2b, a^4b\}$ is a generalized Cayley subset of $T_{4n}$ induced by $\alpha$.
From $S^{-1}=\{a^nb, a^{n+2}b, a^{n+4}b\}$, we have $\langle SS^{-1}\rangle=\langle a^2\rangle$ and  $|\langle SS^{-1}\rangle|=n<\frac{1}{2}|T_{4n}|$.
By Theorem \ref{thm3.4}, $GC(T_{4n},S,\alpha)$ is not connected.
\end{proof}

\begin{lemma}\label{lem4.14}
For the dihedral group $D_{4n}(n\geq 2)$, all cubic generalized Cayley graphs of $D_{4n}$ are connected and integral if and only of $n=2$.
\end{lemma}
\begin{proof}
Let $D_{4n}=\langle a,b\mid a^{2n}=b^2=e, bab=a^{-1}\rangle$ and $\alpha:a\mapsto a^{-1}, b\mapsto a^2b$ be an involutory automorphism of $D_{4n}$.
If $n\geq 3$, then $S=\{a^3b, a^{-1}b, ab\}$ is a generalized Cayley subset of $D_{4n}$ induced by $\alpha$.
As  $|\langle SS^{-1}\rangle|=|\langle a^2\rangle|=n<\frac{1}{2}|D_{4n}|$, by Theorem \ref{thm3.4}, $GC(D_{4n},S,\alpha)$ is not connected.
If $n=2$, then all involutory automorphisms of $D_8$ and the corresponding partition of $D_8$  are depicted  in Table \ref{tab2} where we omit $\mho=D_8\setminus (\omega\cup \Omega)$.
\begin{table}[htbp]
\caption{}
\centering
\label{tab2}
\begin{tabular}{llllllllllllll}
\toprule
\makecell[c]{$~\text{Involutions of}~\Aut(D_8)$} & \makecell[c]{$\omega$} & \makecell[c]{$\Omega$} \\
\midrule
\makecell[c]{$\alpha:a\mapsto a, b\mapsto a^2b$} & \makecell[c]{$\{e, a^2\}$} & \makecell[c]{$\emptyset$}
\\
\makecell[c]{$\beta_1:a\mapsto a^{-1}, b\mapsto b$} & \makecell[c]{$\{e, a^2\}$} & \makecell[c]{$\{a, a^3, b, a^2b\}$}
\\
\makecell[c]{$\beta_2:a\mapsto a^{-1}, b\mapsto ab$} & \makecell[c]{$\langle a\rangle$} & \makecell[c]{$\emptyset$}
\\
\makecell[c]{$\beta_3:a\mapsto a^{-1}, b\mapsto a^2b$} & \makecell[c]{$\{e, a^2\}$} & \makecell[c]{$\{a, a^3, ab, a^3b\}$}
\\
\makecell[c]{$\beta_4:a\mapsto a^{-1}, b\mapsto a^3b$} & \makecell[c]{$\langle a\rangle$} & \makecell[c]{$\emptyset$}
\\
\bottomrule
\end{tabular}
\end{table}
According to  Table \ref{tab2},  there is no cubic generalized Cayley graph induced by $\alpha$, $\beta_2$ and $\beta_4$.
In addition, it can be verified that all cubic generalized Cayley graphs induced by  $\beta_1$ and $\beta_3$ are isomorphic to $G_4$ in Figure \ref{fig1}.
Thus, all cubic generalized Cayley graphs of $D_8$ are connected and integral.
\end{proof}

\begin{proposition}\label{cor4.15}
For each of groups  $D_{12}\times \mathbb{Z}_2$, $T_{12}\times \mathbb{Z}_2$ and $D_6\times \mathbb{Z}_4$, there exists  a cubic generalized Cayley graph which is not connected.
\end{proposition}
\begin{proof}
By Lemmas \ref{lem4.13} and \ref{lem4.14}, both $D_{12}$ and  $T_{12}$ have cubic generalized Cayley graphs which are not connected.
Let $GC(G,S,\alpha)$ be a not connected cubic generalized  Cayley graph, where $G=D_{12}$ or $T_{12}$.
Then from Lemmas \ref{lem4.11} and \ref{lem4.12}, $GC(G\times \mathbb{Z}_2, S\times \{1\}, \alpha\times \id)$ is   not  connected.

Note that $GC(\mathbb{Z}_4,\{1\},\iota)$ is not a connected generalized Cayley graph of $\mathbb{Z}_4$ and $GC(D_6, \{b,ab,a^2b\},\beta)$ is a generalized Cayley graph of $D_6$ induced by $\beta: a\mapsto a^{-1}, b\mapsto a^2b$.
Then $GC(D_6\times \mathbb{Z}_4, \{b,ab,a^2b\}\times \{1\}, \beta\times \iota)$ is not a connected generalized Cayley graph of $D_6\times \mathbb{Z}_4$.
\end{proof}

\begin{theorem}
Let $G$ be a finite non-abelian group admitting cubic generalized Cayley graph.
Then all cubic generalized Cayley graphs of $G$ are connected and integral if and only if $G$ is isomorphic to the dihedral groups $D_6$, $D_8$ or the quaternion group $Q_8$.
\end{theorem}
\begin{proof}
Since $G$ is non-abelian, its order cannot be equal to $4$.
By Lemma \ref{lem4.2}, the order of $G$ belongs to$\{6,8,10,12,20,24,30\}$.
We divide the proof into the following cases:

\vspace{2mm}

\textbf{Case 1}: Let $|G|=6$.
Then $G\cong D_6$.
Observe that  the generalized Cayley graphs of $D_6$ are $GC(D_6, \{b, ab, a^2b\}, \alpha_j)$, where $\alpha_j: a\mapsto a^{-1}, b\mapsto a^jb$ for $0\leq j\leq 2$.
It can be verified that these  generalized Cayley graphs are isomorphic to $G_2$ in Figure \ref{fig1}.

\vspace{2mm}

\textbf{Case 2}: Let $|G|=8$.
Then $G\cong D_8$ or $Q_8=\langle a, b\mid a^4=e, a^2=b^2, b^{-1}ab=a^{-1}\rangle$.

Since $\Aut(Q_8)\cong Sym(4)$, the involutory automorphisms of $Q_8$ are divided into two conjugate classes.
Let $\alpha: a\mapsto a^{-1}, b\mapsto ab$ and $\beta: a\mapsto a^{-1}, b\mapsto b$ be two involutory automorphisms in different conjugate classes of $\Aut(Q_8)$.
By Lemma \ref{lem4.0}, it suffices to discuss the generalized Cayley graphs induced by $\alpha$ and $\beta$.
Since $\omega_\alpha(Q_8)=\langle a\rangle$ and $\Omega_\alpha(Q_8)=\emptyset$, there is no cubic generalized Cayley graph induced by $\alpha$.
With respect to $\beta$, we obtain three subsets $\omega_\beta(Q_8)=\{e,a^2\}$, $\Omega_\beta(Q_8)=\{a,a^3,ab,a^3b\}$ and $\mho_\beta(Q_8)=\{b,a^2b\}$.
Thus, there are some cubic generalized Cayley graphs.
It is easy to verify that all cubic generalized Cayley graphs of $Q_8$ induced by $\beta$ are isomorphic to $G_4$ in Figure \ref{fig1}.
Therefore, all cubic generalized Cayley graphs of $Q_8$ are connected and integral.

In addition, by Lemma \ref{lem4.14}, all cubic generalized Cayley graphs of $D_8$ are connected and integral.

\vspace{2mm}

\textbf{Case 3}: Let $|G|=10$. Then $D\cong D_{10}$.
Let $S=\{b, ab, a^4b\}$ be a generalized Cayley subset of $D_{10}$ induced by $\alpha: a\mapsto a^{-1}, b\mapsto b$ and the corresponding graph is displayed in Figure \ref{fig3}.
 Notice that $GC(D_{10}, S, \alpha)$ is  connected  but  it not isomorphic to any graph in Figure \ref{fig1}.
\begin{figure}[htbp]
\centering
\includegraphics[scale=0.5]{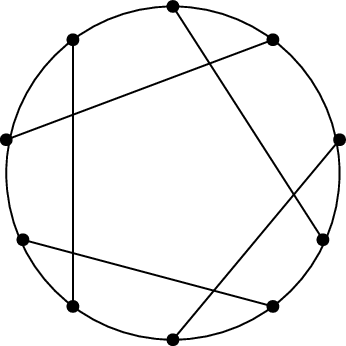}
\caption{}
\label{fig3}
\end{figure}

\vspace{2mm}

\textbf{Case 4}: Let $|G|=12$.
Then $G\cong D_{12}$, $T_{12}$  or $A_4$.
By Lemmas \ref{lem4.13} and \ref{lem4.14}, there exists a cubic generalized Cayley graph of $D_{12}$ and $T_{12}$ which is not connected, respectively.
For $A_4$, let $S=\{(123), (124), (12)(34)\}$, which is a generalized Cayley subset induced by $\alpha:x\mapsto (12)(34)x(12)(34)$ for all $x\in A_4$.
As  $|\langle SS^{-1}\rangle|=|\langle (243)\rangle|< \frac{1}{2}|A_4|$, by Theorem \ref{thm3.4}, $GC(G,S,\alpha)$ is  not connected.

\vspace{2mm}

\textbf{Case 5}: Let $|G|=20$.
Then $G\cong D_{20}$, $T_{20}$ or $F_{5,4}=\langle a, b \mid a^5=b^4=e, b^{-1}ab=a^2\rangle$.
Using Lemmas \ref{lem4.13} and \ref{lem4.14} again,  there is no connected generalized Cayley graphs for both  $D_{20}$ and $T_{20}$.
For $F_{5,4}$, let $\alpha:a\mapsto a^{-1}, b\mapsto b$ be its involutory automorphism  and $S=\{b, b^2, b^3\}$ a generalized Cayley subset of $F_{5,4}$ induced by $\alpha$.
Since $\langle S\rangle\neq F_{5,4}$, by Theorem \ref{thm3.4}, $GC(G,S,\alpha)$ is not connected.

\vspace{2mm}

\textbf{Case 6}: Let $|G|=24$.
Then $G\cong D_{12}\times \mathbb{Z}_2$, $T_{12}\times \mathbb{Z}_2$, $D_6\times \mathbb{Z}_4$, $A_4\times \mathbb{Z}_2$, $Q_8\times \mathbb{Z}_3$, $D_{24}$, $T_{24}$, $Sym(4)$, $D_8\times \mathbb{Z}_3$, $U_{24}=\langle a,b \mid a^8=b^3=e, a^{-1}ba=b^{-1}\rangle$, $V_{24}=\langle a,b \mid a^6=b^4=e, ba=a^{-1}b^{-1}, b^{-1}a=a^{-1}b\rangle$ or $SL(2,3)$.

The groups $ D_{12}\times \mathbb{Z}_2$, $T_{12}\times \mathbb{Z}_2$ and  $D_6\times \mathbb{Z}_4$ are discussed in Proposition \ref{cor4.15}.
As stated in Case $4$,  there exists  a not connected cubic generalized  Cayley graph of $A_4$.
With the same argument in Proposition \ref{cor4.15}, it follows that  there exists  a not connected cubic generalized  Cayley graph of $A_4\times \mathbb{Z}_2$.

For $Q_8\times \mathbb{Z}_3$, let $\alpha_1: (e,1)\mapsto (e,1), (a,0)\mapsto (a^{-1},0), (b,0)\mapsto (b,0)$ be an involutory automorphism  and $S_1=\{(a,0),(a,1),(a,2)\}$  a generalized Cayley subset of $Q_8\times \mathbb{Z}_3$ induced by $\alpha_1$.
Since $\langle S_1\rangle\neq Q_8\times \mathbb{Z}_3$, $GC(Q_8\times \mathbb{Z}_3, S_1, \alpha_1)$ is not connected by Theorem \ref{thm3.4}.

For $D_{24}$, let $\alpha_2: a\mapsto a^5, b\mapsto b$ be its involutory automorphism  and $S_2=\{a^2, a^6, a^{10}\}$  a generalized Cayley subset induced by $\alpha_2$.
Clearly, $\langle S_2\rangle\neq D_{24}$, hence by Theorem \ref{thm3.4}, $GC(D_{24}, S_2, \alpha_2)$ is not connected.

For $T_{24}$, by Lemma \ref{lem4.13}, there exists a not connected cubic generalized Cayley graph of it.

For $Sym(4)$, let $\alpha_3: x\mapsto (12)x(12)$ be its involutory automorphism  and $S_3=\{(12), (13), (23)\}$ a generalized Cayley subset induced by $\alpha_3$.
Then $GC(Sym(4), S_3, \alpha_3)$ is not connected by $\langle S_3\rangle\neq Sym(4)$.

For $D_8\times Z_3$, let  $\alpha_4:(e,1)\mapsto (e,1), (a,0)\mapsto (a^{-1},0), (b,0)\mapsto (a^2b,0)$ be its involutory automorphism  and $S_4=\{(a,0),(a,1),(a,2)\}$  a generalized Cayley subset induced by $\alpha_4$.
Then $GC(D_8\times Z_3, S_4, \alpha_4)$ is not connected by $\langle S_4\rangle\neq D_8\times Z_3$.

For $U_{24}$, let $\alpha_5: a\mapsto a, b\mapsto b^{-1}$ be its involutory automorphism and $S_5=\{a, a^7, a^4\}$ a generalized Cayley subset induced by $\alpha_5$.
Then $GC(U_{24}, S_5, \alpha_5)$ is not connected by $\langle S_5\rangle\neq U_{24}$.

For $V_{24}$, let $\alpha_6: a\mapsto a, b\mapsto b^{-1}$ be its involutory automorphism  and $S_6=\{a, a^5, a^3\}$ a generalized Cayley subset induced by $\alpha_6$.
Then $GC(V_{24}, S_6, \alpha_6)$ is  not connected as $\langle S_6\rangle\neq V_{24}$.

Finally, we show that there is no cubic generalized Cayley graph of $SL(2,3)=\langle A, B\rangle$, where
\begin{equation*}
A=
\begin{bmatrix}
0 & 1\\
2 & 0\\
\end{bmatrix},
B=
\begin{bmatrix}
1 & 1\\
0 & 1\\
\end{bmatrix}.
\end{equation*}
Observe that there are $9$ involutory automorphisms of $SL(2,3)$ which are divided into two conjugate classes.
Let
\begin{equation*}
\alpha_7:
\begin{bmatrix}
0 & 1\\
2 & 0\\
\end{bmatrix}
\mapsto
\begin{bmatrix}
0 & 2\\
1 & 0\\
\end{bmatrix},
\begin{bmatrix}
1 & 1\\
0 & 1\\
\end{bmatrix}
\mapsto
\begin{bmatrix}
0 & 1\\
2 & 2\\
\end{bmatrix}
\end{equation*}
and
\begin{equation*}
\alpha_{8}:
\begin{bmatrix}
0 & 1\\
2 & 0\\
\end{bmatrix}
\mapsto
\begin{bmatrix}
2 & 1\\
1 & 1\\
\end{bmatrix},
\begin{bmatrix}
1 & 1\\
0 & 1\\
\end{bmatrix}
\mapsto
\begin{bmatrix}
1 & 2\\
0 & 1\\
\end{bmatrix}
\end{equation*}
be two involutory automorphisms in different conjugate classes.
For $\alpha_7$, we obtain subsets $\omega_{\alpha_7}(SL(2,3))=\{E, A, A^2, A^3, B^2A^3B, B^2AB\}$ and $\Omega_{\alpha_7}(SL(2,3))=\emptyset$.
Similarly, with respect to $\alpha_{8}$, we have $\omega_{\alpha_{8}}(SL(2,3))=\{E, B, B^2, A^2, AB^2A^3B, AB^2AB, B^2A^2, AB,\linebreak B^2A^3, A^3B, BA^2, B^2A\}$ and $\Omega_{\alpha_{8}}(SL(2,3))=\emptyset$.
It follows that there is no cubic generalized Cayley graph of $SL(2,3)$.

\vspace{2mm}

\textbf{Case 7}: Let $|G|=30$.
Then $G\cong D_{30}, U_{30}=\langle a,b \mid a^{10}=b^3=e, a^{-1}ba=b^{-1}\rangle$ or $D_{10}\times \mathbb{Z}_3$.

Let $\alpha_9: a\mapsto a^4, b\mapsto b$ be an involutory automorphism of $D_{30}$ and $S_9=\{b, a^5b, a^{10}b\}$ a generalized Cayley subset induced by $\alpha_9$.
Then $GC(D_{30}, S_9, \alpha_9)$ is not connected as $\langle S_9\rangle\neq D_{30}$.
Similarly, let $\alpha_{10}: a\mapsto a, b\mapsto b^{-1}$ be an involutory automorphism of $U_{30}$ and $S_{10}=\{a, a^{-1}, a^5\}$ a generalized Cayley subset induced by $\alpha_{10}$.
Then $GC(U_{30}, S_{10}, \alpha_{10})$ is not connected as $\langle S_{10}\rangle\neq U_{30}$.
For the group $D_{10}\times \mathbb{Z}_3$, $\alpha_{11}:(e,1)\mapsto (e,1), (a,0)\mapsto (a^{-1},0), (b,0)\mapsto (a^2b,0)$ is its involutory automorphism  and $S_{11}=\{(a,0),(a,1),(a,2)\}$  is a generalized Cayley subset induced by $\alpha_{11}$.
Since $\langle S_{11}\rangle\neq D_{10}\times \mathbb{Z}_3$, $GC(D_{10}\times \mathbb{Z}_3, S_{11}, \alpha_{11})$ is not connected.

In conclusion, among all non-abelian groups admitting a cubic generalized Cayley graph, the groups $D_6$, $D_8$ and $Q_8$ satisfy that whose all cubic generalized Cayley graphs are connected and integral.
\end{proof}

\end{sloppypar}

\end{document}